\documentclass[12pt]{article}
\usepackage{color}
\usepackage{amsmath}
\usepackage{amsbsy}
\usepackage{mathtools}
\usepackage{bbm}
\usepackage{amstext}
\usepackage{amsthm}
\usepackage{bbold}
\usepackage{amssymb}
\usepackage[dvips]{graphicx}
\usepackage{amsfonts}
\usepackage{caption2}
\usepackage{makeidx}
\usepackage{fancyhdr}
\usepackage{epic,eepic,epsfig}
\usepackage{mdwlist}
\usepackage{amscd}
\usepackage{graphicx}
\usepackage{mathrsfs,wasysym,stmaryrd}
\usepackage[all]{xypic}
\usepackage[knot,poly]{xy}
\usepackage{latexsym}
\usepackage{comment}

\theoremstyle{plain}

\newtheorem{theorem}{Theorem}
\newtheorem{proposition}{Proposition}

\newtheorem{corollary}{Corollary}

\date{}

\begin{document}

\begin{center}
\textbf{\LARGE{On complexity and Jacobian of cone over a graph}}
\vspace{12pt}

{\large\textbf{L.~A.~Grunwald,}}\footnote{{\small\em Sobolev Institute of Mathematics,
Novosibirsk State University, lfb\_o@yahoo.co.uk}}
{\large\textbf{I.~A.~Mednykh,}}\footnote{{\small\em Sobolev Institute of Mathematics,
Novosibirsk State University, ilyamednykh@mail.ru}}
\end{center}

\section*{Abstract}

For any finite graph $G,$ consider a graph $\widetilde{G}$ which is a cone over graph $G.$ In this paper, we study two important invariants of such a cone. Namely, complexity (the number of spanning trees) and the Jacobian of a graph. We prove that complexity of graph $\widetilde{G}$ coincides the number of rooted spanning forests in graph $G$ and the Jacobian of $\widetilde{G}$ is isomorphic to cokernel of the operator $I+L(G),$ where $L(G)$ is Laplacian of $G$ and $I$ is the identity matrix. As a consequence, one can calculate the complexity of $\widetilde{G}$ as $\det(I+L(G)).$

As an application, we establish general structural theorems for Jacobian of $\widetilde{G}$ in the case when
$G$  is  a circulant graph or cobordism of two circulant  graphs.

\smallskip

\noindent
\textbf{Key Words:} spanning tree, spanning forest, circulant graph, Laplacian matrix, cone over graph, Chebyshev polynomial\\
\textbf{AMS classification:} 05C30, 39A10\\

\section{Introduction}

A \textit{spanning tree} in a finite connected graph $G$ is defined as a subgraph of $G$ that contains all vertices of $G$ and has no cycles. The number of spanning trees in the graph $G$ is also called a \textit{complexity} of $G.$ It is a very important graph invariant and along with the pioneers in this field \cite{Cay89}, formulas were found for some special graphs such as the wheel \cite{BoePro}, fan \cite{Hil74}, ladder \cite{Sed69}, M\"obius ladder \cite{Sed70}, lattice \cite{SW00}, prism \cite{BB87} and anti-prism \cite{SWZ16}. However, one of the most significant and general results is the Kirchhoff matrix-tree theorem  \cite{Kir47} which states that the complexity of $G$ can be expressed as the product of nonzero Laplacian eigenvalues of $G,$ divided by the number of its vertices. In this paper, we will also apply the idea \cite{BoePro} of using Chebyshev polynomials for counting various invariants of graphs arose.

Also, no less interesting invariant of a graph is the number of rooted spanning forests in a graph $G.$  According to the classical result \cite{KelChel}, this value can be found as determinant $\det(I + L(G)).$ Here, $L(G)$ is the \textit{Laplacian matrix} of graph $G.$ However, not many explicit formulas are known. One of the first results was obtained by O. Knill \cite{Knill}, who found the analytical formula of the number of rooted spanning forests in the complete graph $K_{n}.$ Some formulas were  obtained for bipartite graphs \cite{JinLin}, cyclic, star, line graphs \cite{Knill} and some others \cite{Sung}. In our previous paper \cite{GrunMed}, we developed a new method for counting rooted spanning forests in circulant graphs. As for the number of unrooted forests, it has a much more complicated structure \cite{Calan}, \cite{Liu}, \cite{Tak}.

Another well-known invariant of a finite graph is \textit{Jacobian group} (also known as the Picard group, critical group, sandpile group, dollar group). This concept was introduced independently by several authors \cite{Cori}, \cite{Baker}, \cite{Biggs}, \cite{BaHaNa}. In particular, the order of the Jacobian group coincides with the number of spanning trees for a graph. This fact is one of the reasons why interest in the Jacobian of a graph is manifested. However, the description of the Jacobian structure remains a difficult task and its structure is known only in several cases \cite{Cori}, \cite{Biggs}, \cite{Lor}, \cite{HouWo}, \cite{Chen}, \cite{MedZin} and \cite{MedMed2}. In this paper we intend to use the result \cite{Med1} about the cokernel structure of Laplacian operator. 

The paper is organized as follows. 

Section~\ref{basic} contains basic definitions and some known results on circulant graphs and circulant matrices. In Section~\ref{cokernel}, we describe general properties of cokernels for $\mathbb{Z}$-linear operators represented by circulant matrices. The main result of Section~\ref{trees} is Theorem~\ref{treestoforests} which asserts   that the number of spanning trees in the cone over a graph $G$ coincides with the number of rooted spanning forests in $G.$ In Section~\ref{jacgcone}, we introduce a notion of the forest group for a graph $G$ defined as the cokernel of $\mathbb{Z}$-linear operator  $I+L(G).$ Then the main result of the section (Theorem~\ref{forest}) states that Jacobian of the cone over a graph $G$ is isomorphic to the forest group of $G.$ Sections \ref{circjac} and \ref{cobjac} are devoted to description of Jacobian groups for the cone over a circulant graph and the cone over cobordism of two circulant graphs respectively. Lastly, in Section~\ref{Exp}, we use the obtained results to calculate Jacobian group and number of spanning trees for cones over some simple families of graphs.    
 
\section{Basic definitions and preliminary facts}\label{basic}

Let $G$ be a finite graph without loops. We denote the vertex and edge set of $G$ by $V(G)$ and $E(G),$ respectively. Given $u, v\in V(G),$ we set $a_{uv}$ to be equal to the number of edges between vertices $u$ and $v.$ The matrix $A=A(G)=\{a_{uv}\}_{u, v\in V(G)}$ is called \textit{the adjacency matrix} of the graph $G.$ The degree $d(v)$ of a vertex $v \in V(G)$ is defined by $d(v)=\sum_{u\in V(G)}a_{uv}.$ Let $D=D(G)$ be the diagonal matrix of the size $|V(G)|$ with $d_{vv} = d(v).$ The matrix $L=L(G)=D(G)-A(G)$ is called \textit{the Laplacian matrix}, or simply \textit{Laplacian}, of the graph $G.$ 

Consider the Laplacian $L(G)$ as a homomorphism $\mathbb{Z}^{n}\to\mathbb{Z}^{n},$ where $n$ is the number of vertices in $G.$ The cokernel $\textrm{coker}\,(L(G))=\mathbb{Z}^{n}/\textrm{im}\,(L(G))$ is an Abelian group. It can be uniquely represented in the form 
$$\textrm{coker}\,(L(G))\cong\mathbb{Z}_{d_{1}}\oplus\mathbb{Z}_{d_{2}}\oplus\cdots\oplus\mathbb{Z}_{d_{n}},$$ where $d_{i}$ satisfy the conditions $d_i\big{|}d_{i+1},\,(1\leq i\leq n-1).$ We note that   $d_{i}=\delta_{i}/\delta_{i-1},$ where $\delta_{i},\,i=1,2,\ldots,n$ is the greatest common divisor of all $i\times i$ minors of matrix $L(G)$ and $\delta_{0}=1.$

Suppose that the graph $G$ is connected, then the groups ${\mathbb Z}_{d_{1}},{\mathbb Z}_{d_{2}},\ldots,{\mathbb Z}_{d_{n-1}}$ --- are finite, and $\mathbb{Z}_{d_{n}}=\mathbb{Z}.$ In this case, we define Jacobian of the graph $G$ as
$$Jac(G)\cong\mathbb{Z}_{d_{1}}\oplus\mathbb{Z}_{d_{2}}\oplus\cdots\oplus\mathbb{Z}_{d_{n-1}}.$$ 
In other words, $Jac(G)$ is isomorphic to the torsion subgroup of $\textrm{coker}\,(L(G)).$

Let $s_1,s_2,\ldots,s_k$ be integers such that $1\leq s_1<s_2<\ldots<s_k\leq\frac{n}{2}.$ The graph $C_{n}(s_1,s_2,\ldots,s_k)$ with $n$ vertices $0,1,2,\ldots,~{n-1}$ is called \textit{circulant graph} if the vertex $i,\, 0\leq i\leq n-1$ is adjacent to the vertices $i\pm s_1,i\pm s_2,\ldots,i\pm s_k\ (\textrm{mod}\ n).$ When $s_k<\frac{n}{2}$ all vertices of a graph have even degree $2k.$ If $n$ is even and $s_k=\frac{n}{2},$ then all vertices have odd degree $2k-1.$ It is well known that the circulant $C_n(s_1,s_2,\ldots,s_k)$ is connected if and only if $\textrm{gcd}\,(s_1,s_2,\ldots,s_k,n)=1.$

We call an $n\times n$ matrix \textit{circulant,} and denote it by $circ(a_0, a_1,\ldots,a_{n-1})$ if it is of the form
$$circ(a_0, a_1,\ldots, a_{n-1})=
\left(\begin{array}{ccccc}
a_0 & a_1 & a_2 & \ldots & a_{n-1} \\
a_{n-1} & a_0 & a_1 & \ldots & a_{n-2} \\
  & \vdots &   & \ddots & \vdots \\
a_1 & a_2 & a_3 & \ldots & a_0\\
\end{array}\right).$$

It easy to see that adjacency and Laplacian matrices of the circulant graph are circulant matrices. The converse is also true. If the Laplacian matrix of a graph is circulant then the graph is also circulant.

Recall \cite{PJDav} that the eigenvalues of matrix $C=circ(a_0,a_1,\ldots,a_{n-1})$ are given by the following simple formulas $\lambda_j=P(\varepsilon^j_n),\,j=0,1,\ldots,n-1,$ where $P(x)=a_0+a_1 x+\ldots+a_{n-1}x^{n-1}$ and $\varepsilon_n$ is an order $n$ primitive root of the unity. Moreover, the circulant matrix $C=P(T),$ where $T_n=circ(0,1,0,\ldots,0)$ is the matrix representation of the shift operator $T_n:(x_0,x_1,\ldots,x_{n-2},x_{n-1})\rightarrow(x_1, x_2,\ldots,x_{n-1},x_0).$

\begin{comment}
During the paper, we will use the basic properties of Chebyshev polynomials. Let $T_n(z)=\cos(n\arccos z)$ and $U_{n-1}(z)={\sin(n\arccos z)}/{\sin(\arccos z)}$ be the Chebyshev polynomials of the first and second kind respectively.

Then $T^{\prime}_n(z)=n\,U_{n-1}(z),\,T_n(1)=1,\,U_{n-1}(1)=n.$ For  $z\neq0$ we have $T_n(\frac{1}{2}(z+z^{-1}))=\frac{1}{2}(z^n+z^{-n}),$ and for $z\neq1$ the identity $(T_n(z)-1)/(z-1)=U_{n-1}^2(\sqrt{(z+1)/2})$ holds.

Also, $T_n(z)$ and $U_{n-1}(z)$ admit the following quantum representation $T_n(z)=({q^n+q^{-n}})/2$ and $U_{n-1}(z)=(q^n-q^{-n})/(q-q^{-1}),$ where $q=z+\sqrt{z^2-1}.$ See \cite{MasHand} for more advanced properties.

By $I_n$ we denote the identity matrix of order $n.$
\end{comment}

\section{Cokernels of linear operators}\label{cokernel}

Let $P(z)$ be a bimonic integer Laurent polynomial. That is
$P(z)=z^p+a_1z^{p+1}+\ldots+a_{s-1}z^{p+s-1}+z^{p+s}$ for some integers
$p,a_1,a_2,\ldots,a_{s-1}$ and some positive integer $s.$ Introduce the
following \textit{companion matrix} $\mathcal{A}$ for polynomial $P(z):$
$\mathcal{A}=\left(\begin{array}{c}\begin{array}{c|c} \mathbb{0} &
I_{s-1}\end{array}\\\hline -1,-a_1,\ldots,-a_{s-1}\\
\end{array}\right),$ where $I_{s-1}$ is the identity $(s-1)\times(s-1)$ matrix.
We  note that  $\mathcal{A}$ is invertible and inverse matrix $\mathcal{A}^{-1}$ is also
an integer matrix.  

Let $\mathbb{A}=\langle\alpha_j,\,j\in\mathbb{Z}\rangle$ be a free Abelian group freely
generated by elements $\alpha_j,\,j\in\mathbb{Z}.$ Each element of $\mathbb{A}$ is a
linear combination $\sum\limits_{j}c_j \alpha_j$ with integer coefficients $c_j.$
Define the shift operator $T:\mathbb{A}\rightarrow\mathbb{A}$ as a
$\mathbb{Z}$-linear operator acting on generators of $\mathbb{A}$ by
the rule $T:\alpha_j\rightarrow\alpha_{j+1},\,j\in\mathbb{Z}.$
Then $T$ is an endomorphism of $\mathbb{A}.$ 

Let $P(z)$ be an arbitrary Laurent
polynomial with integer coefficients, then $A=P(T)$ is also an endomorphism of
$\mathbb{A}.$ Since $A$ is a linear combination of powers of $T,$ the action of
$A$ on generators $\alpha_j$ can be given by the infinite set of linear transformations
$A:\alpha_j\to\sum\limits_{i}a_{i,j}\alpha_i,\, j\in\mathbb{Z}.$ Here all sums
under consideration are finite. We set $\beta_j=\sum\limits_{i}a_{i,j}\alpha_i.$
Then $\textrm{im}\,A$ is a subgroup of $\mathbb{A}$ generated by $\beta_j,\,j\in\mathbb{Z}.$
Hence, $\textrm{coker}\,A=\mathbb{A}/\textrm{im}\,A$ is an abstract Abelian group
$\langle x_i, i\in\mathbb{Z}|\,\sum\limits_{i}a_{i,j}x_i=0,\,j\in\mathbb{Z}\rangle$
generated by $x_i,\,i\in\mathbb{Z}$ with the set of defining relations
$\sum\limits_{i}a_{i,j}x_i=0,\,j\in\mathbb{Z}.$ Here $x_j$ are images of $\alpha_j$
under the canonical homomorphism $\mathbb{A}\rightarrow\mathbb{A}/\textrm{im}\,A.$
Since $T$ and $A=P(T)$ commute, subgroup $\textrm{im}\,A$ is invariant under the
action of $T.$ Hence, the actions of $T$ and $A$ are well defined on the factor
group $\mathbb{A}/\textrm{im}\,A$ and are given by $T:x_j\rightarrow x_{j+1}$ and
$A:x_j\rightarrow\sum\limits_{i}a_{i,j}x_i$ respectively.

This allows to present the group $\mathbb{A}/\textrm{im}\,A$ as follows
$\langle x_i,\,i\in\mathbb{Z}|\,P(T)x_j=0,\,j\in\mathbb{Z}\rangle.$ In a similar
way, given a set $P_1(z),P_2(z),\ldots,P_s(z)$ of Laurent polynomials with integer
coefficients, one can define the group
$\langle x_i,\,i\in\mathbb{Z}|\,P_1(T)x_j=0,P_2(T)x_j=0,\ldots,P_s(T)x_j=0,\,j\in\mathbb{Z}\rangle.$

We will use the following proposition. By $I=I_n$ we denote the identity matrix of order $n.$

\begin{proposition}\label{prop0}Let  $P(z)$ be a bimonic Laurent polynomial with integer coefficients and $\mathcal{A}$ be a companion matrix of $P(z).$  Consider  $L=P(T_n):\mathbb{Z}^n\to \mathbb{Z}^n$ as a $\mathbb{Z}$-linear   operator. Then $$\textrm{coker}\,L\cong \textrm{coker}(\mathcal{A}^{n}-I).$$  
\end{proposition}

\begin{proof}  Since the Laurent polynomial $P(z)$ is bimonic, it can be
represented in the form $P(z)=z^{p}+a_1z^{p+1}+\ldots+a_{s-1}z^{p+s-1}+z^{p+s},$
where $p,s,a_1,a_2,\ldots,a_{s-1}$ are integers and $s>0.$ Then the corresponding companion matrix
$\mathcal{A}$ is
$\left(\begin{array}{c}
\begin{array}{c|c}0 &I_{s-1}\end{array}\\ \hline
-1,-a_1,\ldots,-a_{s-1}
\end{array}\right).$

Let $T$ be the shift operator defined by  
$T:x_j\rightarrow x_{j+1}\,,j\in\mathbb{Z}.$ 
Note that for any $j\in\mathbb{Z}$ the relations $P(T)x_j=0$ can be  written as
$x_{j+s}=-x_{j}-a_1x_{j+1}-\cdots-a_{s-1}x_{j+s-1}.$

 Let $\textbf{x}_j=(x_{j+1},
x_{j+2},\ldots,x_{j+s})^t$ be $s$-tuple of generators $x_{j+1},x_{j+2},\ldots,x_{j+s}.$
Then the relation $P(T)x_j=0$ is equivalent to $\textbf{ x}_{j}=\mathcal{A}\,\textbf{x}_{j-1}.$
Hence, we have $\textbf{x}_{1}=\mathcal{A}\,\textbf{x}_0$ and
$\textbf{x}_{-1}=\mathcal{A}^{-1}\,\textbf{x}_0,$ where $\textbf{x}_0=(x_{1},x_{2},\ldots,x_{s})^t.$
Thus $\textbf{x}_{j} = \mathcal{A}^j\,\textbf{x}_0$ for any $j\in\mathbb{Z}.$ Conversely,
the latter implies $\textbf{x}_{j} = \mathcal{A}\,\textbf{x}_{j-1}$ and, as a consequence,
$P(T)x_j=0$ for all $j\in\mathbb{Z}.$

Consider $\textrm{coker}\,A=\mathbb{A}/\textrm{im}\,A$ as an abstract Abelian group with the
following representation $\langle x_i,i\in\mathbb{Z}\Large| P(T)x_j=0,j\in\mathbb{Z}\rangle.$
We show that $\textrm{coker}\,A\cong\mathbb{Z}^s.$ Indeed,
\begin{eqnarray*}
&&\textrm{coker}\,A=\langle x_i,i\in\mathbb{Z}|\,P(T)x_j=0,j\in\mathbb{Z}\rangle=\\
&&=\langle x_i,i\in\mathbb{Z}\,|\,
x_{j}+a_1 x_{j+1}+\ldots+a_{s-1} x_{j+s-1}+ x_{j+s}=0,\,j\in\mathbb{Z}\rangle\\
&&=\langle x_i,i\in\mathbb{Z}\,|\,(x_{j+1}, x_{j+2},\ldots, x_{j+s})^t=
\mathcal{A}(x_{j},x_{j+1},\ldots,x_{j+s-1})^t,\,j\in\mathbb{Z}\rangle\\
&&=\langle x_i,i\in\mathbb{Z}\,|\,(x_{j+1},x_{j+2},\ldots,x_{j+s})^t=
\mathcal{A}^{j}(x_{1},x_{2},\ldots,x_{s})^t,\,j\in\mathbb{Z}\rangle\\
&&=\langle x_1,x_2,\ldots,x_s\,|\,-\rangle\cong\mathbb{Z}^s.
\end{eqnarray*}
Now, our aim is to  find  cokernel of   $L=P(T_n).$ 
In the operator  notations  
$$\textrm{coker}\,L=\langle x_i,i\in\mathbb{Z}\Large| P(T)x_j=0,\,(T^n-1)x_j,j\in\mathbb{Z}\rangle.$$ 

We  set $B=T^n-1$  and note (\cite{Med1}, Lemma 3.1) that  
$$\textrm{coker}\,L
\cong\textrm{coker}\,A/\textrm{im}(B|_{\textrm{coker}\,A})\cong\textrm{coker}\,(B|_{\textrm{coker}\,A}).$$

We describe the action of the endomorphism $B|_{\textrm{coker}\,A}$ on the
$\textrm{coker}\,A.$ Since the operators $A=P(T)$ and $T$ commute, the action
$T|_{\textrm{coker}\,A}:x_j\to x_{j+1},\,j\in\mathbb{Z}$ on the $\textrm{coker}\,A$
is well defined. First of all, we describe the action of $T|_{\textrm{coker}\,A}$
on the set of generators $x_1,x_2,\ldots,x_s.$ For any $i=1, \ldots, s-1$, we have
$T|_{\textrm{coker}\,A}(x_i)=x_{i+1}$ and $T|_{\textrm{coker}\,A}(x_s)=x_{s+1}=
-x_1-a_1x_2-\ldots-a_{s-2}x_{s-1}-a_{s-1}x_s$. Hence, the action of $T|_{\textrm{coker}\,A}$
on the $\textrm{coker}\,A$ is given by the matrix $\mathcal{A}.$ Considering $\mathcal{A}$
as an endomorphism of the $\textrm{coker}\,A,$ we can write
$T|_{\textrm{coker}\,A}=\mathcal{A}.$ Then,  
$B|_{\textrm{coker}\,A}= \mathcal{A}^n -I$  and $ \textrm{coker}\,L\cong\textrm{coker}\,(\mathcal{A}^n -I).$
 
\end{proof}

\section{The number of spanning trees of a cone over a graph}\label{trees}

The \textit{joint} of graphs $G_{1}$ and $G_{2}$ is called the graph $G=G_{1}*G_{2},$ of order $m+n$, obtained from the disjoint union of $G_{1}$ of order $m,$ and $G_{2}$ of order $n,$ by additionally joining every vertex of $G_{1}$ to every vertex of $G_{2}.$ If $G_{2}=K_{1}$ (the one-vertex graph with no edges) we are going to call the graph $G=G_{1}*K_{1}$ a \textit{cone over graph} $G_{1}.$

Let $G$ be a graph on $n$ vertices. We define $\chi_{G}(\lambda)=\det(\lambda\,I_{n}-L(G))$ as the characteristic polynomial of matrix $L(G),$ which is Laplcaian matrix of graph $G.$ Its extended form is
$$\chi_{G}(\lambda)=\lambda^n+c_{n-1}\lambda^{n-1}+\ldots+c_1\lambda.$$
The theorem by Kelmans and Chelnokov \cite{KelChel} states  that the absolute value of coefficient $c_{k}$ of $\chi_{G}(\lambda)$ coincides with the number of rooted spanning $k$-forests in the graph $G.$ So, the number of rooted spanning forests of the graph $G$ can be found by the formula
\begin{eqnarray}\label{KelmCheln}
f(G)&=&f_{1}+f_{2}+\ldots+f_{n}=|c_{1}-c_{2}+c_{3}-\ldots+(-1)^{n-1}|\\
\nonumber &=&(-1)^{n}\chi_{G}(-1)=\det(I_{n}+L(G)).
\end{eqnarray}

This result was independently obtained by many authors:  see, for example,  \cite{ChebSham},   \cite{Knill},   and \cite{GrunMed}.  

The main result of this section is the following theorem.

\begin{theorem}\label{treestoforests}
The number of spanning trees $\tau(\widetilde{G})$ in the graph $\widetilde{G},$ which is a cone over a graph $G,$ coincides with the number of rooted spanning forests $f(G)$ in the graph $G.$ 
\end{theorem}

\begin{proof}
As a corollary of the well-known Matrix-Tree-Theorem \cite{Kir47}, the number of spanning trees of graph $G$ of order $n,$ can be found by the next formula $\tau(G)=\frac{(-1)^{n-1}}{n}\chi_{G}'(0).$ According to the result by Kelmans (\cite{Kel1}, \cite{Kel2}, see also \cite{Bojan}, Corollary 3.7), for characteristic polynomial of a joint of two graphs $G_{1}$ and $G_{2},$ of order $m$ and $n,$ we have 
$$\chi_{G_{1}*G_{2}}(x)=\frac{x(x-n-m)}{(x-n)(x-m)}\chi_{G_{1}}(x-n)\chi_{G_{2}}(x-m).$$
As a consequence, for a graph $\widetilde{G}=G_{1}*G_{2},$ where $G_{1}=G$ and $G_{2}=K_{1},$ we get
$$\tau(\widetilde{G})=\frac{(-1)^{n}}{n+1}\chi_{\widetilde{G}}'(0)=\frac{(-1)^{n}}{n+1}\chi_{G*K_{1}}'(0)=\frac{(-1)^{m}}{m+1}\left(\frac{x(x-1-n)}{(x-1)(x-n)}\chi_{K_{1}}(x-n)\chi_{G}(x-1)\right)_{x=0}'.$$ 
It is known that $\chi_{K_{1}}(x)=x$, so we obtain
$$\tau(\widetilde{G})=\frac{(-1)^{n}}{n+1}\lim_{x\to0}\frac{(x-1-n)(x-n)\chi_{G}(x-1)}{(x-1)(x-n)}=(-1)^{n}\chi_{G}(-1).$$ By making use of formula (\ref{KelmCheln}) we finish the proof.
\end{proof}

The following corollary gives a convenient way to calculate the complexity of cone over a graph.

\bigskip
\begin{corollary}\label{laplacian} The number of spanning trees in the cone over a graph $G$ is given by the  formula $|\chi_{G}(-1)|,$ where $\chi_{G}(x)$ is the Laplacian characteristic polynomial of $G.$
\end{corollary}

\textbf{Remark to Theorem~\ref{treestoforests}.} There is a natural way to get a one-to-one correspondence between spanning trees in the cone $\widetilde{G}$ and rooted spanning forests in the graph $G.$  

Indeed, consider $\widetilde{G}$ as a joint $G*\{v_0\}$ of $G$ with one-vertex graph $\{v_0\}.$ Let $t$ be a spanning tree in $\widetilde{G}.$ We note that $v_0$ is a vertex of $t.$ Let  $v_0v_j,\,j=1,2,\ldots,k$ be all the edges of graph $t$ coming from vertex $v_0.$ Then $f=t\cap G$ is a spanning forest in $G$ consisting of $k$ trees $t_1,t_2,\ldots,t_k$ chosen in such a way that $v_j$ is a vertex of $t_j.$ So, the pairs $(t_j,v_j),\,j=1,2,\ldots,k$ form a rooted spanning forest in $G.$ In turn, if $(t_j,v_j),\,j=1,2,\ldots,k$ is a rooted spanning forest in $G,$ then the graph $t$ obtained as a union of edges $v_0v_j$ and trees $t_j,\,j=1,2,\ldots,k$ is a spanning tree in $\widetilde{G}.$

\section{Jacobian of a cone over a graph and forest group}\label{jacgcone}
\bigskip

The aim of the current section is to prove the following theorem.

\begin{theorem}\label{forest}
Let $G$ be a graph  on $n$ vertices. Then Jacobian of the cone over $G$ is isomorphic to the cokernel of linear operator $I_n+L(G).$ Here $L(G)$ is a Laplacian matrix of $G$ and $I_n$ is the identity matrix of order $n.$  
\end{theorem}

\begin{proof} 
For any given graph $G$ on $n$ vertices denote by $L(G)$ the Laplacian matrix of the graph $G$ and by $\widetilde{G}$ a graph that is a cone over graph $G.$ It easy to see that the Laplacian matrix of $\widetilde{G}$ can be represented in the following form
$L(\widetilde{G})=\left(\begin{array}{cc}
n & -\mathbb{1}_n \\
-\mathbb{1}_n^{T} & I_{n}+L(G)
\end{array}\right),$ where $I_{n}$ is an identity matrix of order $n$ and $\mathbb{1}_n$ is a vector $(1,1,\ldots,1)$ of length $n.$ To find the Jacobian of the $\widetilde{G}$ we use the following useful relation between the structure of the Laplacian matrix and the Jacobian of a graph \cite{Lor}.
% The proof of this relation can be found in \cite{Lor}. 
 
 Consider the Laplacian $L(\widetilde{G})$ as a homomorphism $\mathbb{Z}^{n+1}\to\mathbb{Z}^{n+1},$ where $n+1$ is the number of vertices in $\widetilde{G}.$ The cokernel $\textrm{coker}\,(L(\widetilde{G}))=\mathbb{Z}^{n+1}/\textrm{im}\,(L(\widetilde{G}))$ --- is an Abelian group. Let $$\textrm{coker}\,(L(\widetilde{G}))\cong\mathbb{Z}_{d_{1}}\oplus\mathbb{Z}_{d_{2}}
\oplus\cdots\oplus\mathbb{Z}_{d_{n+1}}$$
be its Smith normal form satisfying the conditions  $d_{i}\big{|}d_{i+1},\,(1\le{i}\le{n}).$ As the graph $\widetilde{G}$ is connected, the groups $\mathbb{Z}_{d_{1}},\mathbb{Z}_{d_{2}},\ldots,\mathbb{Z}_{d_{n}}$ --- are finite, and $\mathbb{Z}_{d_{n+1}}=\mathbb{Z}.$ Here $d_{i}=\delta_{i}/\delta_{i-1},$ where $\delta_{i},\,i=1,2,\ldots,n+1$ is the greatest common divisor of all $i\times i$ minors of matrix $L(\widetilde{G})$ and $\delta_{0}=1.$ Then,
$$Jac(\widetilde{G})\cong\mathbb{Z}_{d_{1}}\oplus\mathbb{Z}_{d_{2}}
\oplus\cdots\oplus\mathbb{Z}_{d_{n}}$$
is the Jacobian of the graph $\widetilde{G}.$  

To calculate the Smith normal form of a given matrix one can use the following basic operations to convert the matrix to a diagonal form.

\begin{enumerate*}
\item[$1^\circ$.] Add arbitrary integer linear combination of rows to any other row.
\item[$2^\circ$.] Add arbitrary integer linear combination of columns to any other column.
\item[$3^\circ$.] Interchange any two rows or columns.
\end{enumerate*}

The matrix $\left(\begin{array}{cc}
n & -\mathbb{1}_n \\
-\mathbb{1}_n^{T} & I_{n}+L(G)
\end{array}\right)$ is the Laplacian matrix for graph $\widetilde{G}.$ So, the sum of all rows and the sum of all columns in this matrix are zero vectors. Adding all the other rows to the first row we get zero first row. Then we add to the first column the all remained columns to get zero first column. As such, one can easily check that up to operations $1^\circ-3^\circ$ of the above matrix is equivalent to matrix $\left(\begin{array}{cc}
0 & \mathbb{0} \\
\mathbb{0}^{T} & I_{n}+L(G)
\end{array}\right),$ where $\mathbb{0}$ is a zero vector of length $n.$ Therefore, all nonzero elements of the Smith normal form for $L(\widetilde{G})$ coincide with those for matrix $I_{n}+L(G).$
\end{proof}
We note that matrix $I_{n}+L(G)$ is always non-singular. Also, $\textrm{coker}\,(I_{n}+L(G))$ is an Abelian group whose size $\det(I_{n}+L(G))$ is equal to the number of rooted spanning forests in graph $G.$ So, it is natural to call $\textrm{coker}\,(I_{n}+L(G))$ as a \textit{forest group} of $G$ and denote it by $F(G).$ Then, then main  statement of Theorem~\ref{forest} can be rephrased as follows: 
\smallskip

\textit{Jacobian of the cone over a graph $G$ is isomorphic to its forest group $F(G).$}

\section{Jacobian of cone over a circulant graph}\label{circjac}

This section is devoted to Jacobians of a specific class of graphs, namely cones over circulant graphs. We consider two types of circulant graphs $G=C_{n}(s_{1},s_{2},\ldots,s_{k}) $ and $G=C_{2n}(s_{1},s_{2},\ldots,s_{k}, n)$  with even and odd valency of vertices respectively. Denote by $\widetilde{G}$ the cone over  graph $G.$ Using the results of section \ref{jacgcone}, we establish general structural theorems  for $Jac({\widetilde{G}})$  or, equivalently, for the forest group $F(G)=\textrm{coker}(I+L(G)).$

\subsection{Forest group of circulant graph of even valency}

Consider a $2k$-valent circulant graph $G=C_{n}(s_{1},s_{2},\ldots,s_{k}),$ where $1\leq s_{1}<\ldots<s_{k}<\frac{n}{2}.$ Its Laplacian has the form $L(G)=2k I_{n}-\sum\limits_{l=1}^{k}(T_{n}^{s_{l}}+T_{n}^{-s_{l}}),$ where $T_{n}=\textrm{circ}(0,1,0,\ldots,0)$ is $n\times n$ circulant matrix representing the shift operator: $(x_{1},x_{2},\ldots,x_{n-1},x_{n})\rightarrow(x_{2},x_{3},\ldots,x_{n},x_{1}).$ Then the forest group $\textrm{coker}(I+L(G))$ has the following presentation 
$$\langle x_{i},\,i\in\mathbb{Z}\big|(2k+1)x_{j}-\sum\limits_{l=1}^{k}
(x_{j+s_{l}}+x_{j-s_{l}})=0,x_{j+n}=x_{j},\,j\in\mathbb{Z}\rangle.$$
By Proposition~\ref{prop0}, we conclude that $\textrm{coker}(I+L(G))$ is isomorphic to the $\textrm{coker}(\mathcal{A}^n-I),$ where $\mathcal{A}$ is a companion matrix of the Laurent polynomial $2k+1-\sum\limits_{l=1}^{k}(z^{s_{l}}+z^{-s_{l}}).$ Combine this observation with Theorem~\ref{forest}, we get the following result.
\begin{theorem}\label{evencone}Let $\widetilde{G}$  be a cone over the circulant graph $G=C_{n}(s_{1},s_{2},\ldots,s_{k}),$ where $1\leq s_{1}<s_{2}<\ldots<s_{k}<\frac{n}{2}.$ Then $Jac({\widetilde{G}})$ is isomorphic to   $\textrm{coker}(\mathcal{A}^n-I),$ where $\mathcal{A}$ is a companion matrix of the Laurent polynomial $2k+1-\sum\limits_{l=1}^{k}(z^{s_{l}}+z^{-s_{l}}).$ 
\end{theorem}

\subsection{Forest group of circulant graph of odd valency}

Consider a $2k+1$-valent circulant graph of the following form 
$$G=C_{2n}(s_{1},s_{2},\ldots,s_{k},n),\text{ where }1\leq s_{1}<s_{2}<\ldots<s_{k}<n.$$ 
In this case, Laplacian matrix of $G$ is 
$(2k+1)\,I_{2n}-T_{2n}^{n}-\sum\limits_{j=1}^{k}(T_{2n}^{s_{j}}+T_{2n}^{-s_{j}}),$ where 
$T_{2n}=\textrm{circ}(0,1,0,\ldots,0)$ is a $2n\times 2n$ circulant matrix. In order to get the forest group of graph $G$ we have to find $\textrm{coker}(I+L(G)).$ It can be viewed as an infinitely generated Abelian group satisfying the following set of relations  
$$\langle x_{i},\,i\in\mathbb{Z}\big|(2k+2)x_{j}-x_{j+n}-
\sum\limits_{l=1}^{k}(x_{j+s_{l}}+x_{j-s_{l}})=0,\,x_{j+2n}=x_{j},\,j\in\mathbb{Z}\rangle.$$
By making use of the shift operator $T:x_{j}\rightarrow x_{j+1},\,j\in\mathbb{Z}$ we rewrite the later formula as 
$$\langle x_{i},\,i\in\mathbb{Z}\big|(2k+2-T^{n}-\sum\limits_{l=1}^{k}
(T^{s_{l}}+T^{-s_{l}}))x_{j}=0,\,(T^{2n}-1)x_{j}=0,\,j\in\mathbb{Z}\rangle.$$ 
We can increase the list of relations by ones that are linear combinations of elements of a given set. 
One of such combinations is 
$$({T}^{2n}-1)+B(T)(2k+2-{T}^n-\sum\limits_{l=1}^{k}
(T^{s_{l}}+T^{-s_{l}})=(2k+2-\sum\limits_{l=1}^{k}(T^{s_{l}}+T^{-s_{l}}))^{2}-1,$$
where $B(T)=2k+2+{T^n}-\sum\limits_{l=1}^{k}(T^{s_{l}}+T^{-s_{l}}).$ In turn, ${T}^{2n}-1$ is a linear combination of $2k+2-{T}^n-\sum\limits_{l=1}^{k}(T^{s_{l}}+T^{-s_{l}})$ and 
$(2k+2 -\sum\limits_{l=1}^{k}(T^{s_{l}}+T^{-s_{l}}))^{2}-1.$ So, it can be replaced by the latter expression in the group presentation. Hence, $\textrm{coker}(I+L(G))$ admits the following presentation 
$$\langle x_{i},\,i\in\mathbb{Z}\big|(2k+2-T^{n}-\sum\limits_{l=1}^{k}(T^{s_{l}}+T^{-s_{l}}))x_{j}
=0,(2k+2 -\sum\limits_{l=1}^{k}(T^{s_{l}}+T^{-s_{l}}))^{2}-1)x_{j}=0,\,j\in\mathbb{Z}\rangle.$$  

By Proposition~\ref{prop0}, the forest group $\textrm{coker}(I+L(G))$ is isomorphic to the $\textrm{coker}(\mathcal{A}^n-(2k+2)I+\sum\limits_{l=1}^{k}(\mathcal{A}^{s_{l}}+\mathcal{A}^{-s_{l}})),$ where $\mathcal{A}$ is a companion matrix of the Laurent polynomial $(2k+2-\sum\limits_{l=1}^{k}(z^{s_{l}}+z^{-s_{l}}))^2-1.$
Applying Theorem~\ref{forest}, we rewrite the obtained result in the following form.
\begin{theorem}\label{oddcone}Let $\widetilde{G}$  be a cone over the circulant graph  
$$G=C_{2n}(s_{1},s_{2},\ldots,s_{k},n),\,1\leq s_{1}<s_{2}<\ldots<s_{k}<n.$$ Then $Jac({\widetilde{G}})$ is isomorphic to the $\textrm{coker}(\mathcal{A}^n-(2k+2)I+\sum\limits_{j=1}^{k}(\mathcal{A}^{s_{j}}+\mathcal{A}^{-s_{j}})),$ where $\mathcal{A}$ is a companion matrix of the Laurent polynomial $(2k+2-\sum\limits_{j=1}^{k}(z^{s_{j}}+z^{-s_{j}}))^2-1.$
\end{theorem}

\section{Jacobian of a cone over cobordism of two circulant graphs}\label{cobjac}

Let us consider two circulant graphs on $n$ vertices, namely $C_{1}=C_{n}(s_{1,1},s_{1,2},\ldots,s_{1,k})$ and $C_{2}=C_{n}(s_{1,1},s_{1,2},\ldots,s_{1,l}),$ with $k$ and $l$ jumps respectively. Then the cobordism of two circulant graphs $C_{1}$ and $C_{2}$ is a graph $G,$ which is obtained from $C_{1}$ and $C_{2}$ by connecting $i$-vertex of $C_{1}$ with $i$-vertex of $C_{2}.$ The Laplacian matrix of graph $G$ has the form
$\left(\begin{array}{cc}
(2k+1)I_{n}-\sum\limits_{r=1}^{k}(T_{n}^{s_{1,r}}+T_{n}^{-s_{1,r}})&-I_{n}\\
-I_{n}&(2l+1)I_{n}-\sum\limits_{r=1}^{l}(T_{n}^{s_{2,r}}+T_{n}^{-s_{2,r}})
\end{array}\right).$ 

The complexity and other spectral properties of graph $G$ were investigated in \cite{AbrBaiMed}.

Denote by $\widetilde{G}$ the cone  over  graph $G.$ The aim of this subsection is to find cokernel of $I+L(G),$ that is $Jac(\widetilde{G}).$ To do this, we will use two bi-infinite sequences $x_{j},y_{j},\,j\in\mathbb{Z}.$ Then cokernel of the linear operator $I+L(G)$ is isomorphic to the group
\begin{align*}
\langle x_{i},y_{i},i\in\mathbb{Z}\,|\,(2k+2)x_{j}-\sum\limits_{r=1}^{k}
(x_{j+s_{1,r}}+x_{j-s_{1,r}})-y_{j}=0,\,x_{j+n}-x_{j}=0,\\
(2l+2)y_{j}-\sum\limits_{r=1}^{l}(y_{j+s_{2,r}}+y_{j-s_{2,r}})-x_{j}=0,\,y_{j+n}-y_{j}=0,\,j\in\mathbb{Z}\rangle.
\end{align*}
We note that $y_{j}=(2k+2)x_{j}-\sum\limits_{r=1}^{k}(x_{j+s_{1,r}}+x_{j-s_{1,r}})$ is an integer linear combinations of $x_{j},\,j\in\mathbb{Z}.$ Equivalently, in the operator form $y_{j}=(2k+2-\sum\limits_{r=1}^{k}(T^{s_{1,r}}+T^{-s_{1,r}}))x_{j}.$ Then the group above is isomorphic to 
$$\langle x_{i}\,|\,(2k+2 -\sum\limits_{r=1}^{k}(T^{s_{1,r}}+T^{-s_{1,r}}))
(2l+2 -\sum\limits_{r=1}^{l}(T^{s_{2,r}}+T^{-s_{2,r}}))-1)x_{j}=0,\,
(T^n -1)x_{j},\,j\in\mathbb{Z}\rangle.$$
By Proposition~\ref{prop0} and Theorem~\ref{forest}, we get the following result.

\begin{theorem}\label{cobordism} Let $\widetilde{G}$ be a cone over the cobordism graph $G.$  Then Jacobian $Jac(\widetilde{G})$ is isomorphic to the cokernel of linear operator $\mathcal{A}^n-I,$ where $\mathcal
{A}$ is a companion matrix of the Laurent polynomial
$$(2k+2-\sum\limits_{r=1}^{k}(z^{s_{1,r}}+z^{-s_{1,r}}))
(2l+2-\sum\limits_{r=1}^{l}(z^{s_{2,r}}+z^{-s_{2,r}}))-1.$$
\end{theorem}

\section{Examples}\label{Exp}

\bigskip
\noindent
\textbf{$1^{\circ}.$} \textbf{Wheel graph} $W(n).$ The graph $W(n)$ is a cone over cyclic graph $C_{n}=C_{n}(1).$ By Theorem~\ref{treestoforests}, the number of spanning trees $\tau(W(n))$ is equal to the number of rooted spanning forests in $C_{n}$ counting earlier in \cite{GrunMed}. Hence, $\tau(W(n))=2T_{n}(\frac{3}{2})-2.$ See also paper \cite{BoePro} for an alternating proof of this result.

By Theorem~\ref{evencone}, the Jacobian of  wheel graph $Jac(W(n))$ is isomorphic to the cokernel of linear operator $\mathcal{A}^{n}-I_{2},$ where $\mathcal{A}=\{\{0,1\},\{-1,3\}\}$ is a companion matrix of the Laurent polynomial $3-z-z^{-1}.$  Direct calculations leads to the well-known result \cite{Lor}: $Jac(W(n))$ is isomorphic to $\mathbb{Z}_{F_n}\oplus\mathbb{Z}_{5F_n}$ if $n$ is even, and $\mathbb{Z}_{L_n}\oplus\mathbb{Z}_{L_n}$ if $n$ is odd, where $F_n$ and $L_n$ are the Fibonacci and Lucas numbers respectively.
\bigskip

\noindent
\textbf{$2^{\circ}.$} \textbf{The cone over the M\"obius ladder $\widetilde{M}(n).$} Recall that the M\"obius ladder $M(n)$ is circulant graph $C_{2n}(1,n).$ By Theorem~\ref{treestoforests} and Theorem~2 from paper \cite{GrunMed}, the number of spanning trees in the cone over M\"obius ladder $\widetilde{M}(n)$ can be found in the following way
$$\tau(\widetilde{M}(n))=4(T_{n}(\frac{3}{2})-1)(T_{n}(\frac{5}{2})+1).$$

The Jacobian of the cone over  $M(n)$ is isomorphic to the cokernel of linear operator $\mathcal{A}^{n}-4I_{4}+\mathcal{A}+\mathcal{A}^{-1},$ where $\mathcal{A}$ is a companion matrix of the Laurent polynomial $(4-z-z^{-1})^2-1.$

\bigskip
\noindent
\textbf{$3^{\circ}.$} {\bf The cone over prism graph $\widetilde{Pr}(n).$} This graph is a cone over cobordism of two cyclic graphs $C_{n}.$ By arguments similar to those from the proof of Theorem~2 in \cite{GrunMed}, the number of spanning trees of the cone over prism graph $\widetilde{Pr}(n)$ is given by the formula
$$\tau(\widetilde{Pr}(n))=4(T_{n}(\frac{3}{2})-1)(T_{n}(\frac{5}{2})-1).$$

By Theorem~\ref{cobordism}, Jacobian of the cone over prism graph $Pr(n)$ is isomorphic to cokernel of the linear operator $\mathcal{A}^{n}-I_{4},$ where $\mathcal{A}$ is a companion matrix of the Laurent polynomial $(4-z-z^{-1})^2-1.$

\section*{ACKNOWLEDGMENTS}
This work was supported by the Russian Foundation for Basic Research (grant 18-01-00036). The study of the second named author was carried out within the framework of the state contract of the Sobolev Institute of Mathematics (project no. 0314-2019-0007).

\end{document}